\documentclass[11pt,twoside]{article}
\usepackage{amsmath,amsthm}
\usepackage[mathscr]{euscript}
\usepackage{amssymb}

\newtheorem{definition}{Definition}[section]
\newtheorem{theorem}[definition]{Theorem}
\newtheorem{lemma}[definition]{Lemma}
\newtheorem{corollary}[definition]{Corollary}

\def\M{{\mathsf{M}}}
\def\L{{\mathsf{L}}}
\def\G{{\cal G}}
\def\B{\mathsf{B}}
\def\luc{{\mathsf{LUC}}}
\def\zt{{\Lambda}({\luc}(\G)^*)}
\def\rp{\mathsf{RP}}
\def\BLip{\mathsf{BLip_b}}
\def\lu{{{\luc}(\G)^*}}
\def\weakstartweakstar{{w^*\hskip-2mm-\hskip-1mm w^*}}
\def\UMeas{\mathsf{M_u}}
\def\Cont{\mathsf{Cont}}
\def\iin{\!\in\!}

\title{\bf Minimal sets determining the topological centre of
the algebra $\luc(\G)^*$}
\author{Stefano Ferri, Matthias Neufang and Jan
Pachl
\thanks{The first author was supported by the Faculty of
Sciences of Universidad de los Andes, via the
 \emph{Proyecto Semilla:
``Large groups and semigroups and their actions {\rm (II)\/}''\/} and
via a \emph{``Fondo de capacitaci\'on para visitas
academicas''\/}; the second author was partially supported by an NSERC Discovery Grant.
The third author benefited from working in the supportive
environment at the Fields Institute. The support is gratefully acknowl\-edged.
{This paper resulted from a visit of the first author to the
Fields Institute in September 2011 which allowed a meeting between
the three authors. The three authors wish to thank
the Fields Institute for the hospitality.}}}
\date{March 10, 2014 (version 2)}
\begin{document}
\maketitle
\begin{abstract}The study of the Banach algebra $\lu$
associated to a topological group $\G$ has been of
interest in abstract harmonic analysis.
In particular, several authors have studied
the topological centre
$\zt$ of this algebra, which is defined as the set of elements
$\mu\in\lu$ such that
left multiplication by $\mu$ is
$\weakstartweakstar$-continuous.
In recent years several works have appeared in
which it is shown that for a locally
compact group $\G$ it is sufficient to test
the continuity of the left translation by $\mu$ at
just one specific point in order
to determine whether $\mu\in\lu$ belongs to
$\zt$.
In this work we extend some of these
results to a much larger class of groups which
includes many non-locally compact groups as well
as all the locally compact ones. This answers a question
raised by H.G. Dales \cite{dales1}.
We also obtain a corollary about the topological centre
of any subsemigroup of $\lu$ containing
the uniform compactification $\G^{\luc}$ of $\G$. In
particular, we shall prove that there are sets of just
one point determining the topological centre of the
uniform compactification $\G^{\luc}$ itself.
\medskip

\medskip

\end{abstract}

\section{Introduction}

\medskip
Given a topological group $\G$, the dual space
$\lu$ of the space of bounded complex-valued left
uniformly continuous functions has a natural
structure of Banach algebra which has been studied
since the 1970's. In particular, the problem of
describing its topological
centre $\zt$ has been studied by many authors.
The question was first considered by A. Zappa in
\cite{zappa} for abelian groups and for locally compact groups 
was completely solved by A. T.-M. Lau in
\cite{lau}, where it was proved that for any such
group $\G$ the topological centre $\zt$ equals the measure
algebra $\M(\G)$.
The second author of this note in \cite{unified} considered the
topological centre problem, in the non-compact case, for $\mathrm{L}_1(\G)^{**}$ with the first
Arens product and for its quotient $\lu$ and,
using a factorization theorem similar to those appearing in
\cite[Satz 3.6.2]{neufangthesis}
and in \cite{equi}, proved that
${\Lambda}(\mathrm{L}_1(\G)^{**})=\mathrm{L}_1(\G)$ and that $\zt=\M(\G)$.
A similar factorization will be also used in this work.
The same problem has been studied for general
(not necessarily locally compact) groups in \cite{ferrineufang}
and in \cite{ambitable} where it was proved
that for a large class of groups --- which includes all the
locally compact groups as well as many non-locally
compact ones --- the topological centre $\zt$ is the
space of uniform measures on $\G$.

The definition of $\zt$ apparently involves a
requirement on the continuity of certain maps at
every point of $\lu$.
However, it was proved
in \cite{budak} and in \cite{dales} that for locally
compact groups it is possible to determine whether
an element of $\lu$ belongs to $\zt$ by testing
the same type of continuity at just a few specific
points of $\lu$. This led to the definition of a DTC-set
(meaning a set Determining the Topological Centre),
i.e., a set with the property that
it is sufficient to test continuity only at the points of this set
in order to decide whether a given element belongs to $\zt$.

Several flavours of DTC-sets have been considered in the literature,
differing in the type of continuity assumed at the points of the set.
Our main result deals with the DTC-sets in the sense of Definition~\ref{def:DTC}.
We shall prove a one-point DTC-set
theorem for the topological groups that
satisfy a fairly weak cardinality condition
(property ($\dagger$) in Theorem~\ref{th:mainresult}).
This answers a question posed in \cite{dales1}.
From this theorem we shall obtain as
a corollary a result about the topological centre of
any subsemigroup of
$\lu$ that contains
the uniform compactification $\G^{\luc}$
of $\G$; in fact, any subsemigroup that contains
$\G^{\luc}\setminus \G$.
In particular, we shall
prove that there are one-point DTC-sets for the
topological centre of the uniform compactification
$\G^{\luc}$ of $\G$ itself, extending a similar
result given in \cite{budak}.
\par

We start by introducing the basic notation and
terminology used throughout the note.
\par
Let $\G$ be a topological group, here always
assumed to be Hausdorff, with identity $e$.
Given a function $f:\G\longrightarrow X$ (where $X$
can be any range) and $x\iin\G$,
the left translate of $f$ by $x$ is the function $\L_xf$
defined by $\L_xf(y) = f(xy)$ for $y\iin\G$.

We denote by $\rp(\G)$ the set of all
right-invariant continuous pseudometrics on $\G$.
From now on we denote by $\G$ not only the group with its topology but
also the uniform space on the set $\G$ induced by $\rp(\G)$;
since we do not consider here any other uniform
structures on $\G$, this notation will not lead to
any ambiguity.
Then $\luc(\G)$ is the space of bounded
complex-valued uniformly continuous functions on $\G$
with the sup norm.
\par

Given $\nu\in\lu$ and $f\in\luc(\G)$ the function
$\nu\bullet f$, defined by
$$(\nu\bullet f)(x):=\langle
\nu,\L_xf\rangle\qquad(x\in\G),$$
is in $\luc(\G)$ (see for example~\cite{redbook}), i.e. $\luc(\G)$
is left introverted.

This operation induces the convolution operation on
$\lu$, defined by
$$\langle \mu\star \nu,f\rangle:=\langle
\mu,\nu\bullet f\rangle\qquad
(\mu,\nu\in\lu,\;f\in\luc(\G)),$$
which turns $\lu$ into a Banach algebra and
$\luc(\G)$ into a left $\lu$-module.

If we denote by $\delta(x)$ the point evaluation at $x$ ($x\in\G$)
and consider the $w^*$-closure of $\delta[\G]$ in $\lu$,
then it can be proved (see for example \cite{redbook}) that this set with
the induced product
is a semigroup compactification of $\G$ which
topologically coincides with the uniform
compactification of the uniform space $\G$.
This compactification, denoted here by $\G^{\luc}$,
coincides with the spectrum of the C$^*$-algebra $\luc(\G)$.
In the locally compact case this is the largest semigroup compactification of $\G$,
meaning that every other compactification is its quotient.
For a general topological group $\G$ the space $\G^{\luc}$ is the greatest ambit $\G$, 
i.e. the greatest $\G$-flow with a point whose orbit is dense.
In the sequel we identify $\G$ with its image
$\delta[\G]$ in $\G^{\luc}$,
so that $\G\subseteq\G^{\luc}\subseteq\lu$.
More on the subject can be found in
\cite{redbook} and in \cite{pestovbook}.

\begin{definition}
    \label{def:DTC}
Let $S$ be a subsemigroup of $\lu$ with the convolution operation
and the $w^*$-topology.
For $D\subseteq S$ write
\[
\Cont(S,D):=\{\,\mu\in S:
\forall \,\nu_0\iin D
\text{ the mapping } \nu\mapsto \mu\star\nu \text{ on } S
\text{ is continuous at } \nu_0 \}
\]
The \emph{topological centre of $S$} is
$
\Lambda(S):= \Cont(S,S)
$.
The set $D$ is said to be a \emph{DTC-set for $S$} iff
$\Lambda(S)=\Cont(S,D)$.
\end{definition}

In the literature $\Lambda(\lu)$ is often denoted by $Z_t(\lu)$.
Note that DTC-sets are interesting only when the group $\G$ is not precompact:
If $\G$ is precompact then $\Lambda(S)=S$
for every subsemigroup of $\lu$, and thus every subset of $S$
is a DTC-set.

In this paper we deal only with the DTC-sets of Definition~\ref{def:DTC}.
However, other variants of the DTC-set notion are also of interest.
In particular, for $D\subseteq S\cap\G^{\luc}$, if we let
\begin{align*}
\Cont_\G(S,D):= \{&\mu\iin S: \\
& \forall \,\nu_0\iin D
\text{ the mapping } \nu\mapsto \mu\star\nu \text{ on } (S\cap\G) \cup\{\nu_0\}
\text{ is continuous at } \nu_0 \} ,
\end{align*}
then the condition $\Lambda(S)=\Cont_\G(S,D)$ is stronger (more restrictive)
than the condition $\Lambda(S)=\Cont(S,D)$
in Definition~\ref{def:DTC}.
As is explained in~\cite[sec.2]{budak} and~\cite[Ch.12]{dales},
if $\G$ is any non-compact locally compact abelian group then
$\zt\neq\Cont_\G(\lu,\{\nu\})$ for every $\nu\iin\G^{\luc}$
but there exists $\nu_0\iin\G^{\luc}$ such that
$\zt=\Cont(\lu,\{\nu_0\})$.

For any $\Delta\in\rp(\G)$  we define
$$\BLip^+(\Delta):=\{f:\G\longrightarrow
[0,1]:|f(x)-f(y)|\le\Delta(x,y)\text{ for all
}x,y\in\G\},$$
$$\B(\Delta):=\{x\in\G:\Delta(e,x)<1\}.$$
\par
It is known that for a large class of topological
groups $\zt$ coincides with the space
$\UMeas(\G)$ of uniform measures on
the uniform space $\G$. One
of several equivalent definitions of
$\UMeas(\G)$ is that a functional
$\mu\in\lu$ is in $\UMeas(\G)$ if and
only if it is $\G$-pointwise
continuous on $\BLip^+(\Delta)$ for every
$\Delta\in\rp(\G)$. When $\G$ is locally compact,
$\UMeas(\G)$  can be identified with the space
of finite Radon measures on $\G$ (\cite[sec.7.3]{uniformmeasures}).
More about uniform measures,
with references to original sources, may be found
in \cite{uniformmeasures}.
\par

The key step in proving our results will be
a factorization theorem which has its roots in
\cite[Satz 3.6.2]{neufangthesis} and that was later
generalized in a number of ways by many authors
considering problems related to topological centres
(see for example \cite{budak}, \cite{ferrineufang}, \cite{equi}
and \cite{ambitable}). The next
section is devoted just to proving the appropriate
version of the theorem which we shall need in order
to prove our main result.

\section{The factorization theorem}

We start the section by defining some cardinal
numbers associated to
topological groups which will be needed in order to
state our results.

\begin{definition}
{\rm We say that a topological group $\G$ is
\emph{$\kappa$-bounded\/}, where $\kappa$ is an
infinite cardinal, if and only if for every neighbourhood {$U$} of the identity in $\G$
there is a set $A\subseteq\G$ with $|A|\le\kappa$ such that $\G=UA$. 
This is equivalent to saying that for every $\Delta\in\rp(\G)$ there is a set
$A\subseteq\G$ with $|A|\le\kappa$ such that $\G=\B(\Delta)A$. 
We denote by $\mathfrak{B}_{\G}$ the least infinite cardinal for
which $\G$ is $\kappa$-bounded.

Given $\Delta\in\rp(\G)$ we denote by
$\eta^{\sharp}(\Delta)$ the least (finite or
infinite) cardinal $\kappa$ for which there exists
a subset $A\subseteq\G$ of cardinality $\kappa$ and
such that $\G=\B(\Delta)A$
and we denote by
$\eta(\Delta)$ the least cardinal number
$\kappa$ such that there exists a set
$A\subseteq\G$ with $|A|\le\kappa$ and a finite set
$K\subseteq\G$ such that $\G=KB(\Delta)A$.\/}
\end{definition}
We are ready to state and prove the factorization
theorem.
\begin{theorem}\label{th:factorization}
Let $\G$ be a topological group and let
$\Delta_1\in\rp(\G)$ with
$\eta(\Delta_1)={\mathfrak B}_{\G}$, then there
exists a family
$$\{\nu_{\psi}\in\G^{\luc}\setminus\G:\psi\in{\mathfrak
B}_{\G}\}$$
such that for every $\Delta\in\rp(\G)$ such that
$\Delta\ge\Delta_1$ and for every family
$$\{h_{\psi}\in\BLip^+(\Delta):\psi\in{\mathfrak
B}_{\G}\}$$
there is $h\in\BLip^+(2\Delta)$ for which
$$h_{\psi}=\nu_{\psi}\bullet h\qquad\text{for all
}\psi\in{\mathfrak B}_{\G}.$$
\end{theorem}
\begin{proof}
By \cite[Lemma 7]{ambitable}, $\G$ has a
$\Delta_1$-dense subset $D$ of cardinality
${\mathfrak B}_{\G}$. Denote by ${\cal P}_f(D)$ the
set of all finite subsets of $D$ and set
$A:={\mathfrak B}_{\G}\times{\cal P}_f(D)$.
\par
By \cite[Lemma 8]{ambitable} there are
$x_{(\psi,K)}\iin\G$ for $(\psi,K)\iin A$ such that
$\Delta_1(Kx_{(\psi,K)},Lx_{(\varphi,L}))>1$ whenever
$(\psi,K)\neq(\varphi,L)$ are elements of
$A$.
\par
For every $\psi\in{\mathfrak B}_{\G}$ let
$\nu_{\psi}$ be a cluster point of the net
$(x_{(\psi,K)})_{K\in{\cal P}_f(D)}$, where ${\cal
P}_f(D)$ is ordered by reversed inclusion.
\par
Define a real-valued function $u_K$ in the
variable $x\in\G$ for every $K\in{\cal P}_f(D)$ by:
$$u_K(x):=(1-2\Delta_1(x,K))^+.$$
Then we have that $u_K\in\BLip^+(2\Delta_1)$ and
that $\displaystyle{\lim_K u_K=1}$ pointwise.
\par
Now take any $\Delta\in\rp(\G)$ with
$\Delta\ge\Delta_1$ and a family of functions
$\{h_{\psi}\in\BLip^+(\Delta):\psi\in{\mathfrak
B}_{\G}\}$. For every $x\in\G$ there is at most
one $(\psi,K)\in A$ such that
$u_K(xx_{(\psi,K)}^{-1})\neq 0$. We can then define
$$h(x):=\sup\{h_{\psi}(xx_{(\psi,K)}^{-1})\wedge
u_K(xx_{(\psi,K)}^{-1}):(\psi,K)\in A\}$$
and we have that $h\in\BLip^+(2\Delta)$. Take
$x\in\G$ and $\psi\in{\mathfrak B}_{\G}$. By
density, there is $y\in D$ with
$\Delta_1(x,y)\le\frac{1}{2}$. For every
$K\in{\cal P}_f(D)$ with
$y\in K$ we have that $h_{\psi}(x)\wedge
u_K(x)=h(xx_{(\psi,K)})=x_{(\psi,K)}\bullet h(x)$,
hence
$$h_{\psi}(x)=\lim_Kh_{\psi}(x)\wedge
u_K(x)=\lim_Kx_{(\psi,K)}\bullet h(x).$$
By \cite[Lemma 19]{ambitable} the mapping
$\nu\mapsto\nu\bullet h$ is continuous from
$\G^{\luc}$ to $\BLip^+(\Delta)$ with the
$\G$-pointwise topology, hence
$h_{\psi}(x)=\nu_{\psi}\bullet h(x)$.
Finally, $\nu_{\psi}\notin\G$ because
from $\nu_{\psi}\in\G$ we would get a contradiction by
choosing $h_{\psi}=0$ and $h_{\varphi}=1$ for
$\psi\neq\varphi$.
\end{proof}

\section{DTC-sets for $\zt$ and $\Lambda(\G^{\luc})$}
We are finally ready to give the main results of
this note. We begin with a technical lemma.

\begin{lemma}\label{lemma:technical}
Let $\G$ be a $\kappa$-bounded topological group
(where $\kappa$ is an infinite cardinal). The
following properties of a functional
$\mu\in\lu$ are equivalent:
\begin{itemize}
\item[{\rm (i)\/}] $\mu\in\UMeas(\G)$.
\item[{\rm (ii)\/}] If $\Delta\in\rp(\G)$ and
$(h_{\psi})_{\psi\in\Psi(\kappa)}$ is a net
in $\BLip^+(\Delta)$ indexed by the
set $\Psi(\kappa):={\cal P}_f(\kappa)\times\omega$
(ordered by $(K,i)\le(L,j)$ if and only if
$K\subseteq L$ and $i\le j$) which converges
pointwise to $0$, then $0$ is a cluster point of
the net $(\mu(h_{\psi}))_{\psi}$.
\item[{\rm (iii)\/}] The restriction of $\mu$ to
$\BLip^+(\Delta)$ is $\G$-pointwise continuous at
$0\in\BLip^+(\Delta)$ for every
$\Delta\in\rp(\G)$.
\end{itemize}
\end{lemma}
\begin{proof}
Evidently, (i) implies (ii).
\par
To prove that (ii) implies (iii), take any $\mu$
that does not have
the property stated in (iii). There is
$\Delta\in\rp(\G)$ for which the restriction of
$\mu$ to $\BLip^+(\Delta)$ is not $\G$-pointwise
continuous at $0$. By \cite[Lemma 7]{ambitable},
$\G$ has a $\Delta$-dense subset $D$ with
cardinality smaller or equal than $\kappa$. Fix a
surjection $\alpha:\kappa\longrightarrow D$. For
every $\psi=(K,i)\in\Psi(\kappa)$ let $U_{\psi}$
be the $D$-pointwise neighbourhood
$$\{f\in\BLip^+(\Delta):f(\alpha(x))<\frac{1}{i+1}\text{
for every }x\in K\}$$
of $0$ in $\BLip^+(\Delta)$. There is
$\varepsilon>0$ such that for every
$\psi\in\Psi(\kappa)$ there is $h_{\psi}\in
U_{\psi}$ for which $|\mu(h_{\psi})|>\varepsilon$.
Since the $\G$-pointwise and the $D$-pointwise
topology coincide on $\BLip^+(\Delta)$, the net
$(h_{\psi})_{\psi}$ converges $\G$-pointwise to
$0$. Hence $\mu$ does not have property (ii).
\par
In order to prove that (iii) implies (i) take any
net $(f_{\gamma})_{\gamma}$ in $\BLip^+(\Delta)$
converging $\G$-pointwise to a function
$f\in\BLip^+(\Delta)$. Then the functions
$(f_{\gamma}-f)^+$ and $(f_{\gamma}-f)^-$ are in
$2\BLip^+(\Delta)$ and converge $\G$-pointwise to
$0$.
\end{proof}

\par
We are ready to state the main result of this
paper, answering the question raised in \cite{dales1}.
\begin{theorem}\label{th:mainresult} Let $\G$ be a topological group
with the following property:
\begin{itemize}
\item[{\rm ($\dagger$)}] There exists $\Delta_0\in\rp(\G)$
such that $\eta^{\sharp}(\Delta_0)=\mathfrak{B}_{\G}$.
\end{itemize}
Then there are $\nu\in\G^{\luc}\setminus\G$ and a
net $(\nu_{\gamma})_{\gamma\in\Gamma}$ in
$\G^{\luc}\setminus\G$ such that:
\begin{itemize}
\item[{\rm ({\bf 1}.)}]
$\displaystyle{\lim_{\gamma\in\Gamma}\nu_{\gamma}=\nu}$,
with the limit taken in $\G^{\luc}$; and
\item[{\rm ({\bf 2}.)}] if $\mu\in\lu$ and
$\displaystyle{w^*\hskip-2mm
-\hskip-1mm\lim_{\gamma\in\Gamma}\mu\star
\nu_{\gamma}=\mu\star\nu}$ then
$\mu\in\UMeas(\G)$.
\end{itemize}
\end{theorem}
\begin{proof}
Since $\G$ has property ($\dagger$), by \cite[Theorem 5]{ambitable} there is
$\Delta_1\in\rp(\G)$ such that
$\eta(\Delta_1)={\mathfrak B}_{\G}$.
Write $\Psi:={\cal P}_f({\mathfrak B}_{\G})\times\omega$
and note that $|\Psi|=|{\mathfrak B}_{\G}|$.
Let
$\{\nu_{\psi}\in\G^{\luc}\setminus\G:\psi\in\Psi \}$ be as
in Theorem \ref{th:factorization}
with $\Psi$ in place of ${\mathfrak B}_{\G}$.
The net
$(\nu_{\psi})_{\psi}$ has a subnet
$(\nu_{\gamma})_{\gamma}$ converging to a limit
$\nu\in\G^{\luc}$.
\par
Take any $\mu\in\lu$ such that
$w^*\hskip-1mm-\lim_{\gamma}\mu\star\nu_{\gamma}=\mu\star\nu$
in $\lu$. Take any $\Delta\in\rp(\G)$ and a
net $(h_{\psi})_{\psi}$ in $\BLip^+(\Delta)$
indexed by $\Psi$ and
converging pointwise to $0$. By Theorem
\ref{th:factorization} there is
$h\in\BLip^+(2\Delta)$ such that
$h_{\psi}=\nu_{\psi}\bullet h$ for all
$\psi\in\Psi$. By \cite[Lemma 19]{ambitable}, the
mapping $\nu\mapsto\nu\bullet h$ is continuous from
$\G^{\luc}$ to $\BLip^+(2\Delta)$ with the
$\G$-pointwise topology, hence $\nu\bullet h=0$.

Since
$$\lim_{\gamma}\mu(\nu_{\gamma}\bullet
h)=\lim_{\gamma}\mu\star\nu_{\gamma}(h)=\mu\star\nu(h)=\mu(\nu\bullet
h)=0$$
and $(\mu(\nu_{\gamma}\bullet h)\}_{\gamma}$ is a
subnet of $(\mu(h_{\psi}))_{\psi}$, we have that
$0$ is a cluster point of this last net and so
$\mu\in\UMeas(\G)$ by Lemma
\ref{lemma:technical}.
\end{proof}
Note that if $\G$ is precompact then
$\eta^{\sharp}(\Delta)$ is finite for every
$\Delta\in\rp(\G)$ and it is known that in this
case $\zt=\UMeas(\G)$. Thus property
($\dagger$) implies that $\G$ is not precompact.
\par
As a direct corollary to Theorem
\ref{th:mainresult} we obtain the following result.
\begin{corollary}
Let $\G$ be a topological group
with property \emph{($\dagger$)},
and let $S$ be a subsemigroup of $\lu$ such that
$S\supseteq\G^{\luc}\setminus\G$.
Then there exists $\nu_0\in S$ such that
$$
\Lambda(S)= \UMeas(\G) \cap S
= \{\mu\in S:
\text{ the mapping } \nu\mapsto \mu\star\nu \text{ on } S
\text{ is continuous} \text{ at } \nu_0 \}.
$$
\end{corollary}
\begin{proof}
By~\cite[Prop. 4.2]{ferrineufang} or~\cite[Cor. 9.36]{uniformmeasures} we have
$\UMeas(\G) \cap S \subseteq \Lambda(S)$.
By Theorem~\ref{th:mainresult} there exists $\nu_0\in S$ such that
$\Cont(S,\{\nu_0\})\subseteq\UMeas(\G) \cap S$,
and obviously $\Lambda(S)\subseteq\Cont(S,\{\nu_0\})$.
\end{proof}

Thus, for any $\G$ with property ($\dagger$), any subsemigroup $S$ of $\lu$ containing
$\G^{\luc}\setminus\G$ has a one-point DTC-set.

It is easy to see that every non-compact
locally compact group has property ($\dagger$):
Simply take any $\Delta_0\in\rp(\G)$ such that
$\B(\Delta_0)$ is relatively compact.
Thus Theorem \ref{th:mainresult} and its corollary
generalize the recent
results of Budak, I\c sik and Pym \cite{budak}, who
proved the same for non-compact locally compact
groups, and therefore the existence of one-point DTC-sets for
$\lu$ and for $\G^{\luc}$ for such groups.

Many other (not necessarily locally compact)
groups also have property ($\dagger$):
From the definition,
if ${\mathfrak B}_{\G}=\aleph_0$ and $\G$ is not precompact
then $\G$ has property ($\dagger$);
and if ${\mathfrak B}_{\G}$ is a successor cardinal
then $\G$ has property ($\dagger$).
\par
It is an open problem whether the property ($\dagger$)
may be omitted in Theorem~\ref{th:mainresult} or in its corollary
(for non-precompact topological groups).
\par
\par

\bigskip
\hrule
\bigskip

\noindent\textsc{Stefano Ferri, Departamento de Matem\'aticas,
Universidad de los Andes, Carrera 1 18A--10, Bogot\'a D.C.,
Colombia.  Apartado A\'ereo 4976.}

\noindent E-mail: {\tt stferri@uniandes.edu.co}
\medskip

\noindent\textsc{Matthias Neufang, School of Mathematics and
Statistics, Carleton University, Ottawa (ON), Canada K1S 5B6 and
Universit\'e de Lille--1, UFR de Math\'e\-matiques, 59655 Villeneuve
d'Ascq, France}.

\noindent E-mail: {\tt mneufang@math.carleton.ca},
{\tt Matthias.Neufang@math.univ-lille1.fr}
\medskip

\noindent\textsc{Jan Pachl, Fields Institute, 222
College Street, Toronto (ON), Canada M5T 3J1.}

\noindent E-mail: {\tt jpachl@fields.utoronto.ca}
\end{document}